\documentclass{amsart}

\usepackage[bookmarksnumbered,colorlinks, plainpages]{hyperref}
\hypersetup{colorlinks=true,linkcolor=red, anchorcolor=green, citecolor=cyan, urlcolor=red, filecolor=magenta, pdftoolbar=true}

\usepackage{amsmath, amsthm, amscd, amsfonts, amssymb, graphicx, color, enumerate, xcolor, mathrsfs, latexsym}

\theoremstyle{plain}
\newtheorem{theorem}{Theorem}[section]
\newtheorem{lemma}[theorem]{Lemma}
\newtheorem{proposition}[theorem]{Proposition}
\newtheorem{corollary}[theorem]{Corollary}

\theoremstyle{definition}

\newcommand{\calA}{\mathcal{A}}
\newcommand{\calI}{\mathcal{I}}
\newcommand{\calP}{\mathcal{P}}

\newcommand{\set}[1]{\left\{#1\right\}}

\newcommand{\calD}{\mathcal{D}}

\begin{document}
\title[Moore-Penrose inverses of incidence matrices]{Moore-Penrose inverses of incidence matrices}

\author{Ali Azimi}
\address{Department of Mathematics and Applied Mathematics, Xiamen University Malaysia, 43900, Sepang, Selangor Darul Ehsan, Malaysia}
\email{ali.azimi61@gmail.com,\ ali.azimi@xmu.edu.my}
\thanks{The first author was supported by a Xiamen University Malaysia Research Fund (Grant No.
XMUMRF/2023-C11/IMAT/0023).}

\author[R. B. Bapat]{R. B. Bapat}
\address{Indian Statistical Institute, New Delhi, 110016, India}
\email{rbb@isid.ac.in}

\author{M. Farrokhi D. G.}
\address{Research Center for Basic Sciences and Modern Technologies (RBST), Institute for Advanced Studies in Basic Sciences (IASBS), Zanjan 45137-66731, Iran}
\email{m.farrokhi.d.g@gmail.com,\ farrokhi@iasbs.ac.ir}

\subjclass[2020]{Primary 05B20, 15A09; Secondary 05A30, 15A03.}
\keywords{Moore-Penrose inverse, incidence matrices, vector space, design}

\begin{abstract}
We present explicit formulas for Moore-Penrose inverses of some families of set inclusion matrices arising from sets, vector spaces, and designs. 
\end{abstract}
\maketitle
\section{Introduction}
The incidence algebra $\calA_\calP$ of a locally finite poset $\calP$  is the associative algebra of all real-valued functions on $\calP\times\calP$ that vanish on pairs $(p,q)$ with $p\not\preccurlyeq q$. The addition being the ordinary sum of functions while the multiplication $f*g$ of two elements $f,g$ of $\calA_\calP$ is defined as
\[(f*g)(x,y)=\sum_{x\leq z\leq y}f(x,y)g(y,z).\]
The zeta function $\zeta_\calP$ of $\calP$ (or $\calA_\calP$) is the constant function $\zeta_\calP(x,y)=1$ on all $(x,y)$ satisfying $x\preccurlyeq y$. The zeta function $\zeta_\calP$ is invertible in the algebra $\calA_\calP$. The inverse of $\zeta_\calP$ is called the M\"{o}bius function $\mu_\calP$ of $\calP$ generalizing the classical M\"{o}bius function from number theory as well as inclusion-exclusion principle from set theory. M\"{o}bius function of posets has studied since 1935 by Weisner \cite{lw-1, lw-2} and Hall \cite{ph} for the lattice of subgroups of finite groups, and later, in a general setting, by Rota \cite{gr} in 1964. 

The incidence matrix $\calI_\calP$ of $\calP$ is the $0-1$ matrix whose rows and columns are indexed by elements of $\calP$ with $(p,q)$ entry being $1$ if $p\preccurlyeq q$ and $0$ otherwise. The matrix $\calI_\calP$ is actually the matrix representation of the zeta function of $\calP$. Accordingly, $\calI_\calP$ is invertible and its inverse is the matrix representation of the M\"{o}bius function of $\calP$.

Beside incidence matrices $\calI_\calP$ of posets, their submatrices, we still call incidence matrices, are also of particular interest when studying relationships between two classes of objects that appear naturally in various contexts like design theory, coding theory, algebraic graph theory, representation theory, finite geometry, etc., see for instance \cite{mb-ps, dbc-ps-qx, nh, rmw1982, rmw1990}. Set inclusion incidence matrices arising from subsets of a set or subspaces of a vertex space are typical and most accessible examples of incidence matrices having many applications in combinatorial contexts, in particular design theory and extremal set theory \cite{tb-dj-hl, pf, rmw1982, rmw1990}. 

Given a set $S$ with $n$ elements and $0\leq r\leq c\leq n$, the \textit{set inclusion incidence matrix} $M=M(n;r,c)$ of $S$ is an $\binom{n}{r}\times\binom{n}{c}$ matrix defined as follows: the rows and columns of $M$ are indexed by $r$-subsets and $c$-subsets of $S$, respectively, with $(R,C)$-entry equal to $\delta_{R\subseteq C}$, where
\[\delta_{R\subseteq C}:=\begin{cases}
1,&R\subseteq C,\\
0,&R\nsubseteq C
\end{cases}\]
is the Kronecker delta. The set inclusion matrix $M(n;r,c)$ of a vector space $V$ of dimension $n$ over a finite field with $q$ elements is defined analogously with rows and columns are indexed by $r$-dimensional and $s$-dimensional subspaces of $V$, respectively. In this case, $M(n;r,c)$ is an $\binom{n}{r}_q\times\binom{n}{c}_q$ matrix with $\binom{n}{i}_q$ being the Gaussian binomial coefficient. The above definitions can be extended to set inclusion matrices of sub-algebras of any algebraic structures as well.

A \textit{generalized inverse} of an $m \times n$ matrix $A$ is an $n \times m$ matrix $G$ satisfying $AGA=A$. The \textit{Moore-Penrose inverse} of $A$, denoted by $A^+$, is a generalized inverse of $A$ satisfying the following three more equations:
\[A^+AA^+=A^+,\quad (AA^+)^T=AA^+,\quad (A^+A)^T=A^+A,\]
where $M^T$ denotes the transpose of a matrix $M$ (see \cite{em, rp}). It is known that any two Moore-Penrose inverses of a matrix are equal and that any complex matrix has a Moore-Penrose inverse (see \cite{slc-cdm, isr-gre} for basic properties of Moore-Penrose inverses). Notice that Moore-Penrose inverses need not exist over any field.

Wilson \cite{rmw1990} studies set inclusion matrices and gives a formula for their rank over fields of prime characteristics. Extending the results of Wilson, Bapat \cite{rbb2000} gives implicit formulas for the Moore-Penrose inverse of set inclusion matrices over fields of characteristic zero. Moreover, a necessary and sufficient condition for these matrices to admit Moore-Penrose inverses over fields of prime characteristics as well as an implicit formula for their Moore-Penrose inverses (when exist) is offered. 

The aim of this paper is to give explicit combinatorial formulas for Moore-Penrose inverses of several classes of incidence matrices. In Section 2, we introduce set inclusion incidence matrices and give their Moore-Penrose inverses. Likewise, in Section 3, we introduces set includion incidence matrices arising from vector spaces and give their Moore-Penrose inverses by extending the results of Section 2. Using information on designs, we also give a simple formula for Moore-Penrose inverses of incidence matrices of designs in Section 4.
\section{Set inclusion incidence matrices}
Let $r$, $c$, and $n$ be non-negative integers satisfying $0\leq r\leq c\leq n$. The set inclusion incidence matrix $M=M(n;r,c)$ is an $\binom{n}{r}\times\binom{n}{c}$ matrix defined as follows: the rows and columns of $M$ are indexed by $r$-subsets and $c$-subsets of $[n]:=\{1,\ldots,n\}$, respectively, with $(R,C)$-entry equal to $\delta_{R\subseteq C}$. Let $M^*$ be the matrix whose rows and columns are indexed by $c$-subsets and $r$-subsets of $[n]$, respectively, and its $(C,R)$-entry is given by
\[M^*_{C,R}=(-1)^{r-i}\frac{\binom{c-i-1}{r-i}}{\binom{N-r}{c-r}\binom{N-c}{r-i}},\]
where $N=\max\{n, r+c\}$ and $i=|R\cap C|$. We show that $M^*$ is the Moore-Penrose inverse of $M$. Indeed, $M^*$ has more interesting properties described in the next lemma. To this end, we apply the following result of Ruiz.
\begin{proposition}[\cite{smr}]\label{Ruiz's identity}
Let $n\geq0$ and $p(x)$ be a polynomial of degree less than $n$. Then
\[\sum_{i=0}^n(-1)^i\binom{n}{i}p(i)=0.\]
\end{proposition}
\begin{lemma}\label{set inclusion incidence matrix: MM^*=I or M^*M=I}
With the above notation, we have
\begin{itemize}
\item[(i)]$MM^*=I$ if $n\geq r+c$.
\item[(ii)]$M^*M=I$ if $n\leq r+c$.
\end{itemize}
\end{lemma}
\begin{proof}
(i) Let $R$ and $R'$ be $r$-subsets of $[n]$ and put $k:=|R\cap R'|$. A simple verification shows that
\[(MM^*)_{R,R'}=\sum_{i=k}^r (-1)^{r-i}\frac{\binom{r-k}{i-k}\binom{n-2r+k}{c-r-i+k}\binom{c-i-1}{r-i}}{\binom{n-r}{c-r}\binom{n-c}{r-i}}.\]
If $R=R'$, then $k=r$ and $(MM^*)_{R,R'}=1$. Now, assume that $R\neq R'$, i.e. $k<r$. Then
\begin{align*}
(MM^*)_{R,R'}&=\frac{(c-r)(-1)^r}{(r-k)\binom{n-r}{r-k}}\sum_{i=k}^r(-1)^i\binom{c-i-1}{r-k-1}\binom{r-k}{i-k}\\
&=\frac{(c-r)(-1)^{r-k}}{(r-k)\binom{n-r}{r-k}}\sum_{j=0}^{r-k}(-1)^j\binom{c-k-j-1}{r-k-1}\binom{r-k}{j}=0
\end{align*}
by Proposition \ref{Ruiz's identity} for $\binom{c-k-j-1}{r-k-1}$ is a polynomial of order less than $r-k$ in $j$.

(ii) Let $C$ and $C'$ be $c$-subsets of $[n]$ and put $k:=|C\cap C'|$. Since $n\leq r+c$, we have $N=r+c$ and $2c-n\leq k\leq c$. A simple verification shows that
\[(M^*M)_{C,C'}=\sum_{i=\max\{0,r+k-c\}}^{\min\{k,r\}}(-1)^{r-i}\frac{\binom{k}{i}\binom{c-k}{r-i}\binom{c-i-1}{r-i}}{\binom{c}{r}\binom{r}{i}}.\]
If $C=C'$, then $k=c$ and $(M^*M)_{C,C'}=1$. Now, suppose that $C\neq C'$, i.e. $k<c$. Then
\begin{align*}
(M^*M)_{C,C'}&=\frac{(c-r)(-1)^r}{(c-k)\binom{c}{k}}\sum_{i=\max\set{0,k+r-c}}^{\min\set{k,r}}(-1)^i\binom{c-i-1}{c-k-1}\binom{c-k}{r-i}\\
&=\frac{c-r}{(c-k)\binom{c}{k}}\sum_{j=\max\set{0,r-k}}^{\min\set{r,c-k}}(-1)^j\binom{c-r+j-1}{c-k-1}\binom{c-k}{j}.
\end{align*}
From $2c-k=|C\cup C'|\leq n\leq r+c$, we get $c-k\leq r$ so that $\min\{r,c-k\}=c-k$. Also, $\binom{c-r+j-1}{c-k-1}=0$ when $j<r-k$. Thus
\[(M^*M)_{C,C'}=\frac{c-r}{(c-k)\binom{c}{k}}\sum_{j=0}^{c-k}(-1)^j\binom{c-r+j-1}{c-k-1}\binom{c-k}{j}=0\]
by Proposition \ref{Ruiz's identity} for $\binom{c-r+j-1}{c-k-1}$ is a polynomial of degree less than $c-k$ in $j$.
\end{proof}
\begin{theorem}\label{set inclusion incidence matrix: mpinv}
$M^+=M^*$.
\end{theorem}
\begin{proof}
First observe that $MM^*$ and $M^*M$ are symmetric for any $(R,R')$-entry of $MM^*$ and $(C,C')$-entry of $M^*M$ depends only on the size of intersections $R\cap R'$ and $C\cap C'$, respectively. On the other hand, by Lemma \ref{set inclusion incidence matrix: MM^*=I or M^*M=I}, either $MM^*=I$ or $M^*M=I$. Therefore $MM^*M=M$ and $M^*MM^*=M^*$ so that $M^*=M^+$ is the Moore-Penrose inverse of $M$.
\end{proof}

Bapat \cite{rbb2000} gives necessary and sufficient conditions for $M$ to have a Moore-Penrose inverse over fields of prime characteristics, and gives an implicit formula for the Moore-Penrose inverse when it exists. It is evident that $M^*$ is a Moore-Penrose inverse of $M$ over a field of characteristic $p$ when denominators of the given entries are non-zero. Hence, $M^*$ yields an explicit formula for the Moore-Penrose inverse of $M$ over a field of characteristic $p$ provided that $p$ does not divides $\binom{N-r}{c-r}\binom{N-c}{0}\cdots\binom{N-c}{r}$, for instance for $p>\max\{n-r,c\}$.
\section{Vector space inclusion incidence matrices}
Let $r$, $c$, and $n$ be non-negative integers satisfying $0\leq r\leq c\leq n$ and $q$ be a prime power. The vector space inclusion incidence matrix $M=M(n,q;r,c)$ is an $\binom{n}{r}_q\times\binom{n}{c}_q$ matrix defined as follows: the rows and columns of $M$ are indexed by $r$--dimensional subspaces and $c$-dimensional subspaces of $GF(q)^n$, respectively, with $(R,C)$-entry equal to $\delta_{R\subseteq C}$. Let $M^*$ be the matrix whose rows and columns are indexed by $c$-dimensional subspaces and $r$-dimensional subspaces of $GF(q)^n$, respectively, and its $(C,R)$-entry is given by
\[M^*_{C,R}=(-1)^{r-i}\frac{\binom{c-i-1}{r-i}_q}{\binom{N-r}{c-r}_q\binom{N-c}{r-i}_q}\cdot\frac{1}{q^{(c-r)(r-i)+\binom{r-i}{2}}},\]
where $N=\max\{n, r+c\}$ and $i=\dim(R\cap C)$. As in the previous section, we aim to show that $M^*$ is the Moore-Penrose inverse of $M$. To this end, we need to recall some notation and prove some preliminary lemmas.

In what follows, for a non-negative integer $n$ and prime power $q$, $[n]_q$ and $[n]_q!$ stand for the $q$-integer $(q^n-1)/(q-1)$ and the $q$-factorial $[n]_q\cdots[1]_q$, respectively. Also, for $0\leq m\leq n$, the $q$-binomial $[n]_q!/[m]_q![n-m]_q!$ is denoted by $\binom{n}{m}_q$. Note that the $q$-binomial $\binom{n}{m}_q$ is defined to be zero if the condition $0\leq m\leq n$ does not hold.
\begin{lemma}
Let $V$ be a vector space of dimension $n$ over a finite field with $q$ elements and $0\leq k\leq i\leq r\leq c\leq n-r$. Let $R$ and $R'$ be subspaces of $V$ of dimension $r$ with a $k$-dimensional intersection. Then the number of subspaces $C\supseteq R$ of $V$ of dimension $c$ such that $C\cap R'$ is $i$-dimensional is equal to 
\[\binom{r-k}{i-k}_q\binom{n-2k+r}{c-r-i+k}_qq^{(c-r-i+k)(r-i)}.\]
\end{lemma}
\begin{proof}
Since $C\supseteq R$ we can work in $V/R$ and assume without loss of generality that $R=0$. Then $C\cap R'$ has dimension $i-k$. Since $R'$ has dimension $r-k$ there are $\binom{r-k}{i-k}$ choices for $C\cap R'$. Passing to the quotient $V/C\cap R'$, we can further assume that $C\cap R'=0$. Then $V$, $C$, and $R'$ have dimensions $n-r-i+k$, $c-r-i+k$, and $r-i$, respectively, with $C\cap R'=0$. The number of subspaces $C$ of $V$ with these properties is equal to 
\[\binom{n-2k+r}{c-r-i+k}_qq^{(c-r-i+k)(r-i)}\]
(see \cite[Lemma 2.6]{mfdg}), from which the result follows.
\end{proof}
\begin{lemma}
Let $V$ be a vector space of dimension $n$ over a finite field with $q$ elements and $0\leq i\leq k\leq r\leq c$. Let $C$ and $C'$ be subspaces of $V$ of dimension $c$ with a $k$-dimensional intersection. Then the number of subspaces $R\supseteq C'$ of $V$ of dimension $r$ such that $R\cap C$ is $i$-dimensional is equal to 
\[\binom{k}{i}_q\binom{c-k}{r-i}_qq^{(r-i)(k-i)}.\]
\end{lemma}
\begin{proof}
First observe that there are $\binom{k}{i}$ choices for $R\cap C$ as $R\cap C$ is an $i$-dimensional subspace of $C\cap C'$ of dimension $k$. Passing to the quotient $V/R\cap C$, we can assume without loss of generality that $R\cap C=0$. Then $R$, $C\cap C'$, and $C'$, have dimensions $r-i$, $k-i$, and $c-i$, respectively, with $R\cap(C\cap C')=0$. The number of subspaces $R$ of $V$ with these properties is equal to 
\[\binom{c-k}{r-i}_qq^{(r-i)(k-i)}\]
(see \cite[Lemma 2.6]{mfdg}), from which the result follows.
\end{proof}

In order to obtain a $q$-analogue of Proposition \ref{Ruiz's identity}, we need to use Gauss' binomial formula.
\begin{proposition}[Gauss' binomial formula {\cite[5.5]{vk-pc}}]\label{Gauss' binomial formula}
Let $x$, $a$, and $q$ be real numbers and $n\geq1$ be a positive integer. Then
\[(x+a)(x+qa)\cdots(x+q^{n-1}a)=\sum_{i=0}^n\binom{n}{i}_qq^{\binom{i}{2}}a^ix^{n-i}.\]
\end{proposition}
\begin{corollary}\label{Ruiz's q-analogue formula}
Let $n$ be a positive integer. Then
\[\sum_{i=0}^n(-1)^i\binom{n}{i}_q\cdot\frac{1}{q^{i(n-i)+\binom{i}{2}}}\cdot p_m(q^i)=0\]
for all polynomials $p_m(x)$ of degree $0\leq m<n$.
\end{corollary}
\begin{proof}
The result follows by replacing $x$ by $-q^ma$ ($0\leq m<n$) in Gauss' binomial formula and assuming that $p_m(x)=x^m$.
\end{proof}

Now, we can proceed to show that $M^*$ is actually the Moore-Penrose inverse of $M$ by showing that the results of Lemma \ref{set inclusion incidence matrix: MM^*=I or M^*M=I} remain unchanged in the case of vector space inclusion incidence matrices.
\begin{lemma}\label{set inclusion incidence matrix: MM^*=I or M^*M=I}
With the notation above, we have
\begin{itemize}
\item[(i)]$MM^*=I$ if $n\geq r+c$;
\item[(ii)]$M^*M=I$ if $n\leq r+c$.
\end{itemize}
\end{lemma}
\begin{proof}
(i) Let $R$ and $R'$ be subspaces of $GF(q)^n$ of dimension $r$ and put $k:=\dim(R\cap R')$. A simple verification shows that
\[(MM^*)_{R,R'}=\sum_{i=k}^r (-1)^{r-i}\frac{\binom{r-k}{i-k}_q\binom{n-2r+k}{c-r-i+k}_q\binom{c-i-1}{r-i}_q}{\binom{n-r}{c-r}_q\binom{n-c}{r-i}_q}\cdot\frac{1}{q^{(i-k)(r-i)+\binom{r-i}{2}}}.\]
If $R=R'$, then $k=r$ and $(MM^*)_{R,R'}=1$. Now, assume that $R\neq R'$, i.e. $k<r$. Then
\begin{align*}
(MM^*)_{R,R'}&=\frac{[c-r]_q(-1)^r}{[r-k]_q\binom{n-r}{r-k}_q}\sum_{i=k}^r(-1)^i\binom{c-i-1}{r-k-1}_q\binom{r-k}{i-k}_q\cdot\frac{1}{q^{(i-k)(r-i)+\binom{r-i}{2}}}\\
&=\frac{[c-r]_q(-1)^{r-k}}{[r-k]_q\binom{n-r}{r-k}_q}\sum_{j=0}^{r-k}(-1)^j\binom{c-k-j-1}{r-k-1}_q\binom{r-k}{j}_q\cdot\frac{1}{q^{j(r-k-j)+\binom{r-k-j}{2}}}\\
&=\frac{[c-r]_q(-1)^{r-k}}{[r-k]_q\binom{n-r}{r-k}_q}\sum_{l=0}^{r-k}(-1)^l\binom{c-r+l-1}{r-k-1}_q\binom{r-k}{l}_q\cdot\frac{1}{q^{l(r-k-l)+\binom{l}{2}}}=0
\end{align*}
by Corollary \ref{Ruiz's q-analogue formula} for $\binom{c-r+l-1}{r-k-1}_q$ is a polynomial of order less than $r-k$ in $q^l$.

(ii) Let $C$ and $C'$ be subspaces of $GF(q)^n$ of dimension $c$ and put $k:=\dim(C\cap C')$. Since $n\leq r+c$, we have $N=r+c$ and $2c-n\leq k\leq c$. A simple verification shows that
\[(M^*M)_{C,C'}=\sum_{i=\max\{0,r+k-c\}}^{\min\{k,r\}}(-1)^{r-i}\frac{\binom{k}{i}_q\binom{c-k}{r-i}_q\binom{c-i-1}{r-i}_q}{\binom{c}{r}_q\binom{r}{i}_q}\cdot\frac{1}{q^{(c-r-k+i)(r-i)+\binom{r-i}{2}}}.\]
If $C=C'$, then $k=c$ and $(M^*M)_{C,C'}=1$. Now, suppose that $C\neq C'$, i.e. $k<c$. Then
\begin{align*}
(M^*M)_{C,C'}&=\frac{[c-r]_q(-1)^r}{[c-k]_q\binom{c}{k}_q}\sum_{i=\max\set{0,k+r-c}}^{\min\set{k,r}}(-1)^i\binom{c-i-1}{c-k-1}_q\binom{c-k}{r-i}_q\cdot\frac{1}{q^{(c-r-k+i)(r-i)+\binom{r-i}{2}}}\\
&=\frac{[c-r]_q}{[c-k]_q\binom{c}{k}_q}\sum_{j=\max\set{0,r-k}}^{\min\set{r,c-k}}(-1)^j\binom{c-r+j-1}{c-k-1}_q\binom{c-k}{j}_q\cdot\frac{1}{q^{(c-k-j)j+\binom{j}{2}}}.
\end{align*}
From $2c-k=\dim(C+C')\leq n\leq r+c$, we get $c-k\leq r$ so that $\min\{r,c-k\}=c-k$. Also, $\binom{c-r+j-1}{c-k-1}=0$ when $j<r-k$. Thus
\[(M^*M)_{C,C'}=\frac{[c-r]_q}{[c-k]_q\binom{c}{k}_q}\sum_{j=0}^{c-k}(-1)^j\binom{c-r+j-1}{c-k-1}_q\binom{c-k}{j}_q\cdot\frac{1}{q^{(c-k-j)j+\binom{j}{2}}}=0\]
by Corollary \ref{Ruiz's q-analogue formula} for $\binom{c-r+j-1}{c-k-1}$ is a polynomial of degree less than $c-k$ in $q^j$.
\end{proof}

Utilizing the same line of proof as in Theorem \ref{set inclusion incidence matrix: mpinv} we can prove the following.
\begin{theorem}\label{set inclusion incidence matrix: mpinv}
$M^+=M^*$.
\end{theorem}

As in the previous section, $M^*$ yields an explicit formula for the Moore-Penrose inverse of $M$ over a field of characteristic $p$ provided that $p$ does not divides $\binom{N-r}{c-r}_q\binom{N-c}{0}_q\cdots\binom{N-c}{r}_qq$, for instance for any $p>\max\{n-r,c\}$ not dividing $q$.
\section{$t$-Design inclusion incidence matrices}
A $t-(v,k,\lambda)$ design $\calD$ is a collection of blocks as $k$-subsets of a $v$-element set, say $[v]$, such that every $t$-subset of $[v]$ belongs to $\lambda$ blocks. It is known that every $t-(v,k,\lambda)$ design is an $s-(v,k,\lambda_s)$, where
\[\lambda_s=\lambda\cdot\frac{\binom{v-s}{t-s}}{\binom{k-s}{t-s}}\]
for all $0\leq s\leq t$. Let $b:=\lambda \binom{v}{t}/\binom{k}{t}$ be the number of blocks in $\calD$. For any $0\leq s\leq t$, let $M_s:=M_s(t,v,k,\lambda)$ be the $\binom{v}{s}\times b$ matrix defined as follows: rows and columns of $M$ are indexed by $s$-subsets of $[v]$ and blocks, respectively, with $(R,C)$-entry equal to $\delta_{R\subseteq C}$. It is a well-known result that $M_1M_1^T=(\lambda_1-\lambda_2)I+\lambda_2J$, where $I$ and $J$ are the identity matrix and the matrix of all $1$'s with appropriate sizes, respectively. Using this fact, we can simply compute the Moore-Penrose inverse of the point-block incidence matrix $M_1$ of block designs.
\begin{proposition}
We have $M_1^+=M_1^T((\lambda_1-\lambda_2)I+\lambda_2J)^{-1}$. More precisely,
\[(M_1^+)_{B,u}=\begin{cases}
\frac{1}{\lambda_1},&u\in B,\\
-\frac{1}{\lambda_1}\cdot\frac{k-1}{v-k},&u\notin B.
\end{cases}\]
\end{proposition}
\begin{proof}
We know from \cite[Lemma 4.4]{rbb} that the $n\times n$ matrix $xI+yJ$ is an invertible matrix if and only if $x(x+ny)\neq0$, from which it follows that $X:=(\lambda_1-\lambda_2)I+\lambda_2J$ is an invertible matrix. Since $M_1(M_1^TX^{-1})=I$ we observe that 
\[M_1(M_1^TX^{-1})M_1=M_1\quad\text{and}\quad(M_1^TX^{-1})M_1(M_1^TX^{-1})=M_1^TX^{-1}.\]
On the other hand, $(M_1^TX^{-1})M_1$ is symmetric as $X$ is a symmetric matrix. Thus $M_1^TX^{-1}=M_1^+$. Now, we compute the entries of $M_1^+$. Let $u$ be a point and $B$ be a block. First observe that the inverse of an invertible matrix of the form $xI+yJ$ is equal to $\frac{1}{x}I-\frac{y}{x(x+ny)}J$. For $u\notin B$, we have
\begin{align*}
(M_1^+)_{B,u}&=(M_1^TX^{-1})_{B,u}=\sum_{v\in B}(X^{-1})_{v,u}\\
&=-\frac{k\lambda_2}{(\lambda_1-\lambda_2)(\lambda_1-\lambda_2+v\lambda_2)}=-\frac{1}{\lambda_1}\cdot\frac{k-1}{v-k}
\end{align*}
as $\lambda_2=(k-1)/(v-1)\lambda_1$. Also, for $u\in B$, we get
\begin{align*}
(M_1^+)_{B,u}&=(M_1^TX^{-1})_{B,u}=\sum_{v\in B}(X^{-1})_{v,u}\\
&=\frac{1}{\lambda_1-\lambda_2}-\frac{k\lambda_2}{(\lambda_1-\lambda_2)(\lambda_1-\lambda_2+b\lambda_2)}\\
&=\frac{1}{\lambda_1-\lambda_2}-\frac{1}{\lambda_1}\cdot\frac{k-1}{v-k}=\frac{1}{\lambda_1},
\end{align*}
as required
\end{proof}

Unexpectedly, the Moore-Penrose inverse of matrices $M_s$ when $s\geq 2$ not only depends on the parameters $t$, $v$, $k$, $\lambda$, but also depends on the internal structure of the designs. Indeed, a computer search reveals that for every set of small parameters $t$, $v$, $k$, $\lambda$ for which all block designs are classified (see design theory database \cite{lhs}) the matrix $M_s^+$ of any design receive uniquely determined set of entries with few exceptions.

\end{document}